\documentclass[a4paper,11pt]{amsart}

\usepackage{amsfonts}
\usepackage{xcolor}
\usepackage{amsmath}
\usepackage{amssymb}
\usepackage{amsthm}

\newtheorem{theorem}{Theorem}[section]
\newtheorem{lemma}[theorem]{Lemma}
\newtheorem{corollary}[theorem]{Corollary}

\newtheorem{ex}[theorem]{Example}
\newtheorem{prop}[theorem]{Proposition}
\newcommand{\C}{\mathbb C}

\newcommand{\N}{\mathbb N}

\DeclareMathOperator{\rk}{rk}

\usepackage{fancyhdr}
\begin{document}

\title{On non-formality of homogeneous spaces}

\author{Zofia St\c{e}pie\'{n}}
\address{
School of Mathematics, West Pomeranian University of Technology\\
al. Piast\'{o}w 48/49, 70-310 Szczecin, Poland}
\email{stepien@zut.edu.pl}

\begin{abstract}
One of the interesting and important rational homotopy properties of a topological space $X$ is that of {\em formality}.
In this paper we prove the non-formality property of some family homogeneous spaces.
\end{abstract}

\maketitle

\section{Introduction}

The path-connected topological space $X$ is called {\em formal} if there is a chain of morphisms  of the form 
\[(A_{PL}(X),d)\rightarrow \ldots \leftarrow\ldots\rightarrow\ldots\leftarrow (H(A_{PL}(X),d),0)\]
inducing isomorphisms in cohomology. Here $A_{PL}(X)$ denotes the algebra of polynomial differential forms on $X.$
Homogeneous spaces constitute a very well-studied and important class of topological spaces. 
There are some classes of compact homogeneous spaces which are
well-known to be formal, for example symmetric spaces (see \cite{GHV}), $k-$symmetric spaces (see \cite{KT}, \cite{S}),
 compact  homogeneous space $G/H$ with $\rk G = \rk H$ (\cite{B}).
In \cite{GHV} a description of formal homogeneous spaces (referred as Cartan pairs) was given. 
Examples of formal and non-formal spaces can be found, for instance, in \cite{K}, \cite{GHV}, \cite{O}.
 In \cite{A} the author gives several construction principles and 
characterizations for non-formal homogeneous spaces. In \cite{GHV} the  authors  show  that  $SU(6)/SU(3)\times SU(3)$ is non-formal  space.
 In this paper we prove a generalization of this in Section 3. Our method is suited for a variety of applications 
and can produce a lot of new examples.

\section{Basic facts}
In the sequel, we will use the following notation. 
Suppose that $G$ is a connected compact Lie group, $H$ is its subgroup, $S$ is the maximal torus
of the subgroup $H,$ and $T$ is the maximal torus of the group $G.$
We denote Lie algebras corresponding to Lie groups by the corresponding Gothic letters. The symbol
${\frak g}(\C)$ denotes the complexification of a real Lie algebra ${\frak g}.$

The question about formality of homogeneous spaces is settled (in principle) by analysis of their Cartan algebras. 
Let us give the precise definition:
Let $\rk G=l$ and let $\rk H=m$. By the Hopf theorem (see \cite{B}) $H^{*}(G)=\Lambda (z_{1},\ldots , z_{l}),$ the exterior algebra generated by 
 universal transgressive elements $z_{1},$ $\ldots,$ $z_{l}.$ Let $y_{1},$ $\ldots,$ $y_{l},$ be generators of $H^{*}({\bf B}G)$ 
corresponding to $z_{1},$ $\ldots,$ $z_{l}$ by transgression, where ${\bf B}G$ is the classifying space for $G.$ 
The {\em Cartan algebra} of the
homogeneous space $G/H$ is the algebra $C =H^{*}({\bf B } H) \otimes H^{*}(G)$ endowed
with the following differential $d:$

\begin{align*}
&d(1 \otimes z_{i}) =  H^{*}({\bf B}i)(y_{i}) \otimes 1  \ (1 \leq i \leq l),\\
&d(b \otimes 1) = 0 \ \text{for} \ b\in H^{*}({\bf B}H),\\
\end{align*}

where $i: H\hookrightarrow G$ denotes the inclusion.

Denote by $I$ the ideal generated by the $H^{*}({\bf B}i)(y_{i})$ (for $1\leq i \leq l$) in  $H^{*}({\bf B}H).$ 
Let $r$ be the cardinality of a minimal system of homogeneous generators of $I.$ Then $r\geq m$ 
(see \cite{O}, \S 8, $n^{\circ}$ 6 and  Proposition 11.8). 
The number $r-m$ is called the {\em deficiency} of $G/H$ and it is denoted by $df(G/H).$

A basic tool for our study of the formality of homogeneous spaces is the following theorem.
(See \cite{O} Theorem 12.2 and Corollary of 11.10.)

\begin{theorem}
 
 \begin{enumerate} In the above notation, the following conditions hold:
	\item $G/H$ is formal if and only if $df(G/H)=0.$
	\item  $df(G/H)=df(G/S).$
 \end{enumerate}
\end{theorem}

By Theorem 2.1  we can work with $G/S$  instead of $G/H.$
We denote by $W_{G}$ the Weyl group of $\frak{g}$ relative to $\frak{t},$ and  by $S_{W_{G}}$ the ring of $W_{G}$-invariant polynomials on $\frak{t}$ with
real coefficients. We can identify $H^{*}({\bf B} G)$ with $S_{W_{G}}$ (see \cite{B}). Then $H^{*}({\bf B} G)$ is generated by $f_{k_{1}}$ $\ldots$ $f_{k_{l}},$ $j=1,\ldots,l,$ where $f_{k_{j}}$ is a polynomial of  degree $k_{j}$.
Moreover  $H^{*}({\bf B}i):H^{*}({\bf B} H)\rightarrow H^{*}({\bf B} S)$ is the map which  assigns to each polynomial  its restriction to $\frak{s}.$

 A {\em Coxeter transformation} is defined 
as the product of all the reflections in the root system of a compact Lie group. 
Neither the choice of simple roots nor the ordering of reflections in the product affects its conjugacy class. We will denote the Coxeter transformation by $c.$ 
Its order in $W_{G}$ is denoted by $h$ and is called the {\em Coxeter number} of $G$. Let $G$ be simple. 
The eigenvalues of $c$ have the form 
$e^{2\pi i m_{j}/h}$, $j=1,\ldots,l$, where $m_{j}$ are integers.  The numbers $m_{1},$ $\ldots,$ $m_{l},$ are called the
exponents of the Weyl group.

Recall that ${\frak g}$ admits an invariant negative definite symmetric bilinear form which one can extend to Hermitian positive definite form on ${\frak g}(\C).$ Moreover there is the orthonormal basis  $X_{m_{_{1}}},...,X_{m_{l}}$  
of ${\frak t}(\mathbb C)$ which consist of eigenvectors of $c$. Here  $X_{m_{i}}$ is the eigenvector corresponding to the eigenvalue $e^{2\pi i m_{j}/h}.$ 
Recall that and $X_{1}=X_{m_{1}}$ and $X_{m_{l}}=X_{h-1}.$ 
Note that polynomials on $\frak{t}$ extend uniquely to polynomials on $\frak{t}(\C)$ (denoted by the same letters).

The following Proposition can be found in \cite{Bou}. (See \cite{Bou}, Ch.V, \S 6, $n^{\circ}$ 2.)

\begin{prop}
Let $m_{1},\ldots, m_{l}$ be the exponents of a simple connected compact Lie group $G.$ Then any free system of 
homogeneous generators of the ring of $W_{G}$-invariant polynomials on $\frak{t}$ has the degrees $m_{1}+1,$ $\ldots,$ $m_{l}+1.$
\end{prop}

In the proof of the above Proposition the following fact can be found:
\begin{lemma}
Let $G$ be a simple connected compact Lie group. 
 Then the coefficient of $X_{1}^{m_{i}}X_{h-m_{i}}$ in $f_{m_{i}+1}$ is nonzero.
\end{lemma}

\begin{corollary}
Let $G$ be simple connected compact Lie group. Then
\begin{enumerate}
	\item $f_{m_{j}+1}(X_{1})= 0$ for $1\leq j<l-1.$
	\item $f_{m_{l}+1}(X_{1})\neq 0.$
\end{enumerate}
\end{corollary}

\begin{proof}
(1) We have 
\[f_{m_{j}+1}(X_{1})=f_{m_{j}+1}(c(X_{1}))=e^{2\pi i(m_{j}+1)/h}f_{m_{j}+1}(X_{1}).\]

If $1\leq j<l-1,$ then $e^{2\pi i(m_{j}+1)/h}\neq 0.$ Hence $f_{m_{j}+1}(X_{1})=0$ for $1\leq j<l-1.$

(2) Apply Lemma 2.3.

\end{proof}

\section{Results}
Now, we  restrict our attention to $G=SU(n).$ According to the  classical Cartan classification, the simple Lie group $SU(n)$ has the type $A_{n-1}.$ 
Traditionally, the canonical coordinates  on ${\frak t}(\C)$ 
are chosen to be $x_{1},$ $\ldots,$ $x_{n}$ satisfying the condition 
$x_{1}+x_{2}+\ldots+x_{n}=0.$ It is well known that the Weyl group $W_{SU(n)}$ is the group $S_{n}$ of all permutations of the 
coordinates $x_{j}.$ Moreover $S_{W_{SU(n)}}$ is generated by $P_{2}$ $\ldots$ $P_{n},$  where $P_{j}$ are the elementary symmetric polynomials of  degree $j$. We will denote by $\left\langle \left\langle P_{2},\ldots,P_{k}\right\rangle\right\rangle$ the subalgebra 
of $S_{W_{SU(n)}}$ generated by $P_{2}$ $\ldots$ $P_{k}.$
It is easy to see that the automorphism which maps $x_{1}$ to $x_{2},$ $x_{2}$ to $x_{3},$
	$\ldots,$ $x_{n}$ to $x_{1}$ is  a Coxeter transformation. 
	According to our notation  in this case $h=n.$
	Let $\varepsilon =e^{2\pi i/n}$ be a primitive root of unity and 
	$X_{k}=1/\sqrt{n}\left(\left(\varepsilon^{k}\right)^{n-1},\left(\varepsilon^{k}\right)^{n-2},\ldots,\varepsilon^{k},1\right).$ It is easy to  check that 
 the vectors  $X_{1},$ $\ldots,$ $X_{n-1}$ form an orthonormal basis of the eigenvectors of $c$ for ${\frak t}(\C).$ 
We will denote this basis by ${\mathcal B}_{A_{n-1}}.$ We will denote by $p_{j}$  the restriction of $P_{j}$ to ${\frak s}.$

\begin{ex}
 Consider $SU(q+1).$ 
By Corollary 2.4,  $P_{q+1}(X_{1})\neq 0$ and $P_{2}(X_{1})=\ldots=P_{q}(X_{1})=0.$
Therefore $P_{q+1}\notin \left\langle \left\langle P_{2},P_{3},\ldots,P_{q}\right\rangle\right\rangle.$
If q is an odd prime then there is $w\in W_{SU(q+1)}$ such that $w(X_{1})\in {\frak s'}(\C),$ where 
${\frak s'}$ is the maximal abelian subalgebra of ${\frak su}(q-1)\times {\frak su}(2).$
\end{ex}

The following lemma is an immediate consequence of the above example.

\begin{lemma}
Let $H$ be a connected closed Lie subgroup of $SU(n)$ and $q$ be an odd prime.
If ${\frak su}(q-1)\times {\frak su}(2)\subset {\frak h},$
then  
$p_{q+1}\notin \left\langle \left\langle p_{2},p_{3},\ldots,p_{q}\right\rangle\right\rangle.$
\end{lemma}

\begin{theorem}
Let $q$ be an odd prime, $G=SU(nq),$ $H=\underbrace{SU(q)\times \ldots\times SU(q)}_{n},$ then
$G/H$ is not a formal space for $n\geq 2.$
\end{theorem}

\begin{proof}
 Let $B=\left\{X_{k}\in {\mathcal B}_{A_{qn-1}}|k\notin q\N\right\}.$ 
Consider  $\left\langle B\right\rangle,$ the linear span of $B.$ Then 
$\dim \left\langle B\right\rangle=\rk (H)=(q-1)n.$ 
Moreover, observe that for each $1\leq k < qn $ such that $k\notin q\N$ and for each $t=0,1,\ldots,n-1$
we have  $\left(\left(e^{2\pi i/qn}\right)^{k}\right)^{t}+\left(\left(e^{2\pi i/qn}\right)^{k}\right)^{t+n}+
\ldots+\left(\left(e^{2\pi i/qn}\right)^{k}\right)^{t+(q-1)n}=0.$ 
Therefore, there is  $w\in W_{SU(qn)}$ such that $w(X_{k})\in {\frak s}(\C)$ for each $X_{k}\in B.$ Thus, ${\frak s}(\C)\cong \left\langle B\right\rangle.$
 Since $\dim(S)=(q-1)n,$ a minimal system of homogeneous generators of the ideal $(p_{2}, p_{3},\ldots , p_{(q-1)n})$ contains 
at least $(q-1)n$ elements. This means that $df(G/S)\geq 0.$
Lemma 2.3 implies that polynomials $p_{k}$ for $2\leq k\leq qn$ and $k\notin q\N+1$ are without relations in $\frak{s}.$  
The  number of such polynomials is $(q-1)n.$  Moreover, by Lemma 3.2, 
$p_{q+1}\notin \left\langle \left\langle p_{2},p_{3},\ldots,p_{q}\right\rangle\right\rangle.$ 
This and Lemma 2.3 imply that  $p_{q+1}$ and all the 
polynomials $p_{k}$ with $k$ such that $2\leq k\leq qn$ and $k\notin q\N+1$ are without relations in $\frak{s}.$
Hence,  $df(G/S)\geq 1.$ By Theorem 2.1,  $df(G/H)\geq 1$ and $G/H$ is non-formal.
\end{proof}

\begin{lemma}
The space $\frac{SU(12)}{SU(4)\times SU(4)\times SU(4)}$ is non-formal.
\end{lemma}

\begin{proof}

 Let 
	$B=\left\{X_{k}\in {\mathcal B}_{A_{11}}|k\notin 4\N\right\}.$
 Observe that for each $1\leq k\leq 12 $ such that $k\notin 4\N$ and for each $t=0,1,2$
we have  $\left(\left(e^{2\pi i/12}\right)^{k}\right)^{t}+\left(\left(e^{2\pi i/12}\right)^{k}\right)^{t+3}+
\left(\left(e^{2\pi i/12}\right)^{k}\right)^{t+6}+\left(\left(e^{2\pi i/12}\right)^{k}\right)^{t+9}=0.$ 
Therefore, there is  $w\in W_{SU(12)}$ such that $w(X_{k})\in {\frak s}(\C)$ for each $X_{k}\in B.$ Since 
$\dim \left\langle B\right\rangle=\rk (H)=9, $  ${\frak s}(\C)\cong \left\langle B\right\rangle.$ 
Lemma 2.3 implies that polynomials $p_{2},$ $\ldots,$ $p_{12}$ except for $p_{5}$ and $p_{9}$ are without relations in $\frak{s}.$ A straightforward calculation shows that $p_{5}\notin \left\langle \left\langle p_{2},p_{3}\right\rangle\right\rangle.$ 
This and Lemma 2.3 imply that  polynomials $p_{2},$ $\ldots,$ $p_{12}$ except for  $p_{9}$ are without relations in $\frak{s}.$
The  number of such polynomials is $10.$
Hence,  $df(G/S)\geq 1.$ By Theorem 2.1,  $df(G/H)\geq 1$ and $G/H$ is non-formal.
\end{proof}

The next result extends Lemma 3.4.

\begin{theorem}
Let $q$ be an odd prime, $G=SU(4n),$ $H=\underbrace{SU(4)\times \ldots\times SU(4)}_{n},$ then
$G/H$ is not a formal space for $n\geq 2.$
\end{theorem}

\begin{proof}
The proof is similar to the proof of Theorem 3.3.
\end{proof}

\end{document}